\documentclass[runningheads,11pt]{llncs}

\usepackage{amssymb}
\usepackage{graphicx}
\usepackage{mathrsfs}
\usepackage{hyperref}
\usepackage{color}
\usepackage{amsmath}
\usepackage{url}
\usepackage{cite}

\urldef{\mailsa}\path|liliping@amss.ac.cn|    
\newcommand{\keywords}[1]{\par\addvspace\baselineskip
\noindent\keywordname\enspace\ignorespaces#1}

\def\sE{{\mathscr E}}
\def\sF{{\mathscr F}}

\def\bR{{\mathbb{R}}}

\def\bfP{{\mathbf{P}}}

\def\b0{{\textbf{0}}}

\def\<{\langle}
\def\>{\rangle}

\begin{document}


\title{Hunt's hypothesis and symmetrization for general 1-dimensional diffusions}

\titlerunning{Hunt's hypothesis and symmetrization}

%
%


\author{Liping Li${}^{1,2}$%
\thanks{The author is partially supported by NSFC (No. 11688101, No. 11801546 and No. 11931004), Key Laboratory of Random Complex Structures and Data Science, Academy of Mathematics and Systems Science, Chinese Academy of Sciences (No. 2008DP173182), and Alexander von Humboldt Foundation in Germany.\\
Email: \mailsa}%
}
\institute{${}^1$RCSDS, HCMS, Academy of Mathematics and Systems Science,\\ Chinese Academy of Sciences, Beijing, China\\
$\quad$\\
${}^2$Bielefeld University, Bielefeld, Germany}

\authorrunning{L. Li}



%
%

\maketitle

\begin{abstract}
In this paper, we will consider the problem that how far from Hunt's hypothesis (H) to symmetrization for a general 1-dimensional diffusion.  A characterization of (H) involving the classification of points for this diffusion will be first obtained.  Then the main result shows that such a process is symmetrizable,  if and only if (H) holds and a certain family of asymmetric shunt points is empty.  Furthermore, we will also derive the representation of associated Dirichlet forms of general 1-dimensional diffusions under symmetry. 

\end{abstract}

\section{Introduction}\label{Sec-intro}

For a so-called Hunt process, Hunt's hypothesis (H) says that every semi-polar set is polar; see, e.g., \cite{BG68}.  This is one of the most important hypotheses, as well as a long outstanding problem, in probabilistic potential theory.  Rich studies appeared in history to explore whether and when it holds, see, e.g., \cite{H56, GR86}.  Symmetrization usually implies (H).  Except for this classical sufficient condition, most works concentrate on L\'evy processes. Particularly a celebrated Kanda's condition for L\'evy processes admitting asymmetry can also lead to (H), see \cite{K76, K78, R77}.

In this paper we will consider the problem concerning (H) for a general $1$-dimensional diffusion $X$, i.e. a continuous strong Markovian process on $\bR$, which is not necessarily a L\'evy process.  When $X$ is regular in the sense that $\mathbf{P}_x(\sigma_y<\infty)>0$ for any $x,y\in \bR$,  it is obviously symmetric with respect to a so-called speed measure, see, e.g., \cite[Chapter V, \S47]{RW87}, and consequently (H) holds.  Hence the non-regular cases are of main interest.  Thanks to the foundational results for general 1-dimensional diffusions in, e.g., \cite{I06, IM74},  we can classify all points in $\bR$ into different families,  like that of regular points, left/right shunt points, traps, and etc.  A concrete definition is presented in Definition~\ref{DEF21}.  By employing this classification,  we will obtain an equivalent condition,  essentially ruling out two kinds of shunt points as illustrated in Example~\ref{Exa0},  to Hunt's hypothesis (H);  see Theorem~\ref{THM0}. 

As mentioned above, symmetrization usually implies (H), but not vice versa.  Then it is interesting to ask how far from (H) to symmetrization.  An enlightening example, Example~\ref{Exa1},  shows that (H) admits a certain family of asymmtric shunt points, denoted by $\Lambda_{ap}$ in \eqref{definitionofLambdaa},  but symmetrization does not.  Then the main result states roughly that $X$ is symmetrizable, if and only if (H) holds and $\Lambda_{ap}$ is empty. 

Under the symmetry,  the theory of Dirichelt forms is a powerful tool to study Markov processes and related probabilistic potential theory. We refer related terminologies and notations to, e.g., \cite{CF12, FOT11}.  In \cite{LY19}, the author with a co-author has already derived the representation of regular Dirichlet forms associated with general 1-dimensional diffusions.  However in the current case $X$ is only symmetrizable,  the associated Dirichlet form is not necessarily regular.  As a byproduct of the characterization of symmetrization for $X$,  we will further obtain the representation of associated Dirichlet form of $X$ and explore the problem when it becomes regular.  The answer is quite simple: It is regular, if and only if the chosen symmetrizing measure is Radon. 

We need to point out that although the diffusions under consideration in this paper are on $\bR$, all analogical results still hold for those on an interval. 

The rest of this paper is organized as follows.  In \S\ref{SEC2}, we will introduce some terminologies regarding Hunt's hypothesis and symmetrization for a Hunt process.  The section \S\ref{SEC3} is devoted to the classification of all points for a general diffusion process $X$ on $\bR$. Particularly, two useful lemmas for proving the main results will be also presented.  In \S\ref{SEC0}, an equivalent condition to (H) will be obtained. The main results,  characterizing the symmetrization of $X$ outside the family of all traps as well as on $\bR$, will be shown in \S\ref{SEC4} and \S\ref{SEC5}.  Finally the associated Dirichlet forms of $X$ under symmetry will be formulated in \S\ref{SEC6}.

\section{Hunt's hypothesis (H) and symmetrization}\label{SEC2}

Let $M$ be a Hunt process on $E$.  Given a nearly Borel set $A\subset E$,  $\sigma_A:=\{t>0:M_t\in A\}$ is the first hitting time of $A$. A point $x\in E$ is called regular for $A$ provided $\mathbf{P}_x(\sigma_A=0)=1$.  Set $A^r:=\{x \in E: \mathbf{P}_x(\sigma_A=0)=1\}$, i.e. the family of all regular points for $A$.  The following definitions are elementary in probabilistic potential theory:

\begin{definition}
Let $\xi$ be a fully supported $\sigma$-finite positive measure on $E$. A set $A\subset E$ is called 
\begin{itemize}
\item[(1)] A polar set provided that there exists a nearly Borel set $B\supset A$ such that $\mathbf{P}_x(\sigma_B<\infty)=0$ for all $x\in E$;
\item[(2)] A $\xi$-polar set provided that there exists a nearly Borel set $B\supset A$ such that $\mathbf{P}_\xi(\sigma_B<\infty)=0$, i.e.  for $\xi$-a.e. $x\in E$,  $\mathbf{P}_x(\sigma_B<\infty)=0$;
\item[(3)] A thin set provided that there exists a nearly Borel set $B\supset A$ such that $B^r=\emptyset$;
\item[(4)] A semi-polar set provided that $A$ is contained in a union of countably many thin sets. 
\end{itemize}
\end{definition}

 Clearly,  a polar set is $\xi$-polar, and in principle, when $M$ satisfies the so-called absolute continuity condition, then a $\xi$-polar set is also polar.  A polar set is thin, and a thin set is semi-polar.  Hunt's hypothesis (H) raised in \cite{H56}, which plays a crucial role in probabilistic potential theory,  says that
 \begin{description}
 \item[\rm (H)] A semi-polar set is polar. 
\end{description}  
In this paper, we will also consider a weak form of (H): 
\begin{description}
\item[\rm ($\text{H}_\xi$)] A semi-polar set is $\xi$-polar. 
\end{description}
Hunt's hypothesis (H) does not always hold.  For example, it fails for the process $M_t=x+t$, $\mathbf{P}_x$-a.s., on $\bR$. 

\begin{definition}
The process $M$ is called \emph{symmetrizable,} if there exists a fully supported $\sigma$-finite measure $m$ such that 
\[
	\int_E P_t f(x)g(x)m(dx)=\int_E f(x)P_tg(x)m(dx)
\]
for all positive Borel functions $f,g$ on $E$,  where $(P_t)_{t\geq 0}$ is the semigroup of $M$.  Meanwhile $M$ is called symmetric with respect to $m$.  
\end{definition}
 
Hunt claimed in \cite{H56} that the symmetrization of $M$ usually implies (H).  Typical examples include Brownian motions, symmetric $\alpha$-stable processes and etc.   However, symmetrization is not necessary for (H).  Getoor conjectured that except in certain obvious cases where a translation component interferes, essentially all L\'evy processes satisfy (H),  see, e.g., \cite{GR86}.  As mentioned in \S\ref{Sec-intro},  a Kanda's condition for L\'evy processes admitting asymmetry leads to (H). In a celebrated paper \cite{S77},  Silverstein proved that a so-called sector condition implies (H) for general Markov processes. This condition is also a foundation of the theory of non-symmetric Dirichlet forms, see, e.g., \cite{MR92}. 

\section{Classification of points for general 1-dimensional diffusions}\label{SEC3}

Our stage is a diffusion process $X=\{\Omega, \mathcal{F}, \mathcal{F}_t, X_t, \theta_t, \mathbf{P}_x\}$, i.e. a continuous strong Markovian process, on $\bR$ with a ceremony $\Delta$, an isolated point attaching to $\bR$.  We impose that $X$ has no killing inside: For any $x\in \bR$,
\[
	\mathbf{P}_x(X_{\zeta-}\in \bR, \zeta<\infty)=0,
\]
where $\zeta=\inf\{t>0: X_t=\Delta\}$ is the lifetime of $X$. 

Let $\sigma_x:=\sigma_{\{x\}}$ denote the first hitting time of $\{x\}$.  For convenience, we write $x\rightarrow y$ to stand for $\mathbf{P}_x(\sigma_y<\infty)>0$ and $x \leftrightarrow y$ to stand for  $\mathbf{P}_x(\sigma_y<\infty)\mathbf{P}_y(\sigma_x<\infty)>0$.  Particularly, $x\rightarrow -\infty$ (resp. $x\rightarrow \infty$) means \[
	\mathbf{P}_x(X_{\zeta-}=-\infty, \zeta<\infty)>0\quad (\text{resp. }\mathbf{P}_x(X_{\zeta-}=\infty, \zeta<\infty)>0).
\]
For  $x\in \bR$, it follows from Blumenthal 0-1 law that 
\[
	e^\pm=\mathbf{P}_x(\sigma_{x\pm}=0)=0 \text{ or }1,
\]
where $\sigma_{x+}:=\inf\{t>0: X_t>x\}$ and $\sigma_{x-}:=\inf\{t>0:X_t<x\}$.  Hereafter following \cite{IM74, I06}, we classify the points in $\bR$ into several families. 

\begin{definition}\label{DEF21} 
A point $x\in\mathbb{R}$ is called
\begin{itemize}
\item[(1)] regular, {denoted by} $x\in \Lambda_2$, if $e^+=e^-=1$;
\item[(2)] singular, if $e^+e^-=0$;
\item[(3)] left singular, {denoted by} $x\in \Lambda_l$, if $e^+=0$; right singular, {denoted by} $x\in \Lambda_r$, if $e^-=0$;
\item[(4)] left shunt, {denoted by} $x\in \Lambda_{pl}$, if $e^+=0, e^-=1$; right shunt, {denoted by} $x\in \Lambda_{pr}$, if $e^-=0, e^+=1$;
\item[(5)] a trap, {denoted by} $x\in \Lambda_t$, if $e^+=e^-=0$.
\end{itemize}
The sets $\Lambda_2$, $\Lambda_l$, $\Lambda_r$, $\Lambda_{pl}$, $\Lambda_{pr}$ and $\Lambda_t$ stand for the subsets of $\mathbb{R}$ containing all the regular points, left singular points, right singular points, left shunt points, right shunt points and traps of $X$ respectively. The family of all singular points is $\Lambda_l\cup \Lambda_r$.
\end{definition}

Clearly, $\Lambda_2=\left(\Lambda_r\cup \Lambda_l\right)^c$, $\Lambda_{pr}\cap \Lambda_{l}=\emptyset$, $\Lambda_{pl}\cap \Lambda_{r}=\emptyset$ and $\Lambda_r\cap \Lambda_l=\Lambda_t$.  The following lemma concerning these families taken from \cite[Lemma~3.1]{L19},  is elementary to general 1-dimensional diffusions. 

\begin{lemma}\label{LM25}
\begin{itemize}
\item[(1)] Assume $a<b<c$. Then
\[
\begin{aligned}	&\mathbf{P}_a(\sigma_c<\infty)=\mathbf{P}_a(\sigma_b<\infty)\mathbf{P}_b(\sigma_c<\infty),\\
&\mathbf{P}_c(\sigma_a<\infty)=\mathbf{P}_c(\sigma_b<\infty)\mathbf{P}_b(\sigma_a<\infty).
\end{aligned}\]
\item[(2)] A point $b\in \Lambda_r$ (resp. $b\in \Lambda_l$) if and only if $\mathbf{P}_b(X_t\geq b, \forall t)=1$ (resp.  $\mathbf{P}_b(X_t\leq b, \forall t)=1$). Thus $b\in \Lambda_t$ if and only if $\mathbf{P}_b(X_t= b, \forall t)=1$.
\item[(3)] Fix $b\in \Lambda_r$ (resp. $b\in \Lambda_l$). Then for any $a>b$ (resp. $a<b$),
\[
	\mathbf{P}_a(X_t\geq b, \forall t)=1,\quad (\text{resp. }\mathbf{P}_a(X_t\leq b, \forall t)=1).
\]
\item[(4)] Fix $b\in \Lambda_{pr}$ (resp. $b\in \Lambda_{pl}$). Then there exists a point $a>b$ (resp. $a<b$) such that
\[
\mathbf{P}_b(\sigma_a<\infty)>0.
\]
\item[(5)] The left singular set $\Lambda_l$ is closed from the right, i.e. if $x_n\in \Lambda_l$ and $x_n\downarrow x$, $x\in \Lambda_l$. The right singular set $\Lambda_r$ is closed from the left, i.e. if $x_n\in \Lambda_r$ and $x_n\uparrow x$, $x\in \Lambda_r$.
\item[(6)] The regular set $\Lambda_2$ is open. Thus the singular set $\Lambda_r\cup \Lambda_l$ is closed.  
\item[(7)] If each point in an open interval $(a,b)$ is regular, i.e. $(a,b)\subset \Lambda_2$, then for any $x,y\in (a,b)$,
\[
	\mathbf{P}_x(\sigma_y<\infty)\mathbf{P}_y(\sigma_x<\infty)>0.
\]
\end{itemize}
\end{lemma}

Since $\Lambda_2$ is open due to Lemma~\ref{LM25}~(6),  we can write $\Lambda_2$ as a union of at most countably many disjoint open intervals as follows:
\begin{equation}\label{eq:lambda2}
	\Lambda_2=\bigcup_{n=1}^{N\leq \infty} (a_n,b_n). 
\end{equation}
In addition, we present another lemma which will be useful in proving our main results. Though it is elementary, we still give a proof for readers' convenience. 

\begin{lemma}\label{LM1}
\begin{itemize}
\item[(1)] Every regular point $x\in \Lambda_2$ is regular for itself, i.e. $\mathbf{P}_x(\sigma_x=0)=1$. 
\item[(2)] If $\Lambda_r\ni x_n\downarrow x\in \Lambda_{pr}$ (resp.  $\Lambda_l\ni x_n\uparrow x\in \Lambda_{pl}$), then
\begin{equation}\label{eq:31}
	\mathbf{P}_x(\sigma_x<\infty)=0. 
\end{equation}
\item[(3)] If $(a,b)\subset \Lambda_2$, $a\in \Lambda_{pr}$ and $x\nrightarrow a$ for all (or one) $x\in (a,b)$, then $\mathbf{P}_a(\sigma_a<\infty)=0$. 
\item[(4)] Assume that $X$ is symmetric with respect to a fully supported $\sigma$-finite measure $\xi$.  Then for any $x\in \Lambda_{r}$ (resp. $x\in \Lambda_l$),  it holds that $y\nrightarrow x$ for any $y<x$ (resp. $y>x$).  Particularly,  for $x\in \Lambda_t$,  it holds that $y\nrightarrow x$ for any $y\neq x$. 
\end{itemize}
\end{lemma}
\begin{proof}
\begin{itemize}
\item[(1)] Since $\Lambda_2$ is open, there exists an interval $(a,b)\subset \Lambda_2$ such that $x\in (a,b)$.  Then the part process $\hat{X}$ of $X$ on $(a,b)$ (see, e.g., \eqref{eq:hatXt}) is a regular diffusion as in \cite[(45.1)]{RW87}.  This implies that (H) holds for $\hat{X}$ and the polar set for $\hat{X}$ must be empty. Particularly, $\{x\}$ is not thin and we must have $\mathbf{P}_x(\sigma_x=0)=1$.  
\item[(2)] We only consider the case $\Lambda_r\ni x_n\downarrow x\in \Lambda_{pr}$ and the other one is analogical.  Set $\sigma_n:=\sigma_{x_n}$.  For any $t\geq 0$ and $n\geq 1$, set $F(s,\omega):=1_{(-\infty, x_n)}(X_{t-s}(\theta_s \omega))$ for $0\leq s<t$ and $\omega\in \Omega$, and $F(s,\omega):=0$ otherwise. Then it follows from the strong Markovian property that
\[
\begin{aligned}
	\mathbf{P}_x\left(X_t<x_n, t> \sigma_n \right)& =\mathbf{E}_x\left(F(\sigma_n, \theta_{\sigma_n}\omega), t> \sigma_n \right) \\
	&=\mathbf{E}_x \left(\mathbf{E}_x\left(F(\sigma_n, \theta_{\sigma_n}\omega) |\mathcal{F}_{\sigma_n}\right), t> \sigma_n \right) \\
	&=\mathbf{E}_x \left(\mathbf{E}_{x_n}\left(F(s, \omega)\right)|_{s=\sigma_n}, t> \sigma_n \right).
\end{aligned}\]
Note that $\mathbf{P}_{x_n}(X_t<x_n)=0$ by means of Lemma~\ref{LM25}~(2).  We have for any $s< t$,
\[
	\mathbf{E}_{x_n}\left(F(s, \omega)\right)= \mathbf{P}_{x_n}(X_t<x_n)=0. 
\]
This yields $\mathbf{P}_x\left(X_t<x_n, t> \sigma_n \right)=0$.  Hence
\[
	\mathbf{P}_x\left(\bigcap_{t\in \mathbf{Q}_+}\left\{X_t\geq x_n\text{ or }t\leq \sigma_n \right\} \right)=1
\]
where $\mathbf{Q}_+$ is the family of all non-negative rational numbers.  Since all paths of $X$ are continuous,  one can easily obtain 
\[
	\mathbf{P}_x\left(\bigcap_{t\geq 0}\left\{X_t\geq x_n\text{ or }t\leq \sigma_n \right\} \right)=1.
\]
Define $\Omega_n:=\{\omega\in \Omega: \forall t\geq 0, X_t\geq x_n \text{ or }t\leq \sigma_n\}$, which is of probability $1$, i.e. $\mathbf{P}_x(\Omega_n)=1$.  Set $\Omega_0:=\{\omega: \sigma_{x+}=0\}$ and $\bar{\Omega}:=\cap_{n\geq 0}\Omega_n$.  Clearly, $\mathbf{P}_x(\bar\Omega)=1$.  To obtain \eqref{eq:31}, it suffices to show $\sigma_x(\omega)=\infty$ for any $\omega\in \bar{\Omega}$.  Indeed, fix $\omega\in \bar{\Omega}$.  Since $\sigma_{x+}=0$, there exists a constant $t_0>0$ (may depend on $\omega$) such that $X_t(\omega)>x$ for $t\in (0,t_0]$. In addition, $x_n\downarrow x$ tells us that there exists $n$ such that $x_n<X_{t_0}(\omega)$ and thus $\sigma_n(\omega)<t_0$.  It follows from $\omega\in \Omega_n$ that for any $t\geq t_0>\sigma_n(\omega)$,  $X_t(\omega)\geq x_n>x$.  Therefore $X_t(\omega)>x$ for all $t>0$, which leads to  $\sigma_x(\omega)=\infty$. 
\item[(3)] Take $(a,b)\ni x_n\downarrow x$ and set $\sigma_n:=\inf\{t>0: X_t\geq x_n\}$.  Then in the sense of $\mathbf{P}_a$-a.s.,  $\sigma_n\downarrow \sigma_{a+}=0$.  It follows that 
\[
	\mathbf{P}_a(\sigma_a<\infty)=\lim_{n\rightarrow \infty}\mathbf{P}_a(\sigma_a<\infty, \sigma_n\leq \sigma_a).
\]
Set $F(\omega):=1_{[0,\infty)}(\sigma_a(\omega))$ for any $\omega\in \Omega$. 
Since $\sigma_a\neq \sigma_n$ and $\sigma_a=\sigma_n +\sigma_a \circ \theta_{\sigma_n}$ on $\{\sigma_a>\sigma_n\}$,  we have by the strong Markovian property that
\[
\begin{aligned}
	\mathbf{P}_a(\sigma_a<\infty, \sigma_n\leq \sigma_a)&=\mathbf{E}_a\left(F(\theta_{\sigma_n} \omega), \sigma_n<\sigma_a \wedge \infty \right) \\
	&=\mathbf{E}_a\left(\mathbf{E}_a \left(F(\theta_{\sigma_n} \omega)|\mathcal{F}_{\sigma_n}\right),\sigma_n<\sigma_a \wedge \infty \right) \\
	&=\mathbf{E}_a \left(\mathbf{E}_{X_{\sigma_n}}(F), \sigma_n<\sigma_a \wedge \infty \right) \\
	&=\mathbf{P}_a \left(\sigma_n<\sigma_a \wedge \infty \right)\mathbf{E}_{x_n}(F).
\end{aligned}\]
Clearly $\mathbf{E}_{x_n}(F)=\mathbf{P}_{x_n}(\sigma_a<\infty)=0$ due to $x_n\nrightarrow a$.  Eventually we can conclude that $\mathbf{P}_a(\sigma_a<\infty)=0$. 
\item[(4)] See \cite[Lemma~3.2]{L19}. 
\end{itemize}
\end{proof}

\section{Hunt's hypothesis for general 1-dimensional diffusions}\label{SEC0}

At first, we present two simple examples where (H) fails. 

\begin{example}\label{Exa0}
\begin{itemize}
\item[(1)] The simplest example is the process of uniform drift, i.e. $Y_t=x+t$, $\mathbf{P}_x$-a.s., for any $x\in \bR$.  In this example, all points in $\bR$ are right shunt, and clearly, (H) fails.

To be more general,  we claim that if a right accumulation shunt point $x\in \Lambda_{pr}$ can be reached from the left, i.e.  $\Lambda_{pr}\ni x_n\downarrow x\in \Lambda_{pr}$ and $y\rightarrow x$ for some $y<x$,  then (H) fails.  Note that every point $x\in \bR$ for $Y$ is such a shunt point.  Indeed,  Lemma~\ref{LM1}~(2) implies that $\{x\}$ is not regular for itself and thus $\{x\}$ is thin. However $\{x\}$ is not polar due to $y\rightarrow x$. 
\item[(2)] Let $Y^0$ be a diffusion on $\bR$ consisting of two distinct components. The restriction of $Y^0$ to $(-\infty, 0)$ is an absorbing Brownian motion with $0$ being the absorbing boundary, and the restriction to $[0,\infty)$ is a $d$-Bessel process with $d\geq 2$, i.e. $|B_t|$ for a $d$-dimensional Brownian motion $B_t$. 

By a Ikeda-Nagasawa-Watanabe piecing-out construction with the instantaneous distribution $\delta_0$, i.e. the Dirac measure at $0$, as formulated in \cite{INW66},  all paths killed upon hitting $0$ (of $Y^0$) can be resurrected at $0$.  We eventually obtain a conservative diffusion process $Y$ on $\bR$ and particularly $y\rightarrow 0$ for any $y<0$.  Note that $0$ is a right shunt point, while $y\nrightarrow 0$ for any $y>0$.  Hence Lemma~\ref{LM1}~(3) indicates that $\{0\}$ is not regular for itself and thus thin.  However $\{0\}$ is not polar with respect to $Y$ because $y\rightarrow 0$ for $y<0$.  
\end{itemize}
\end{example}

The main result of this section will obtain an equivalent condition, essentially ruling out the two kinds of shunt points in Example~\ref{Exa0},  to Hunt's hypothesis (H).  Before stating it, we need to prepare some notations. 
Following \cite[\S3.5]{IM74},  we say $a$ and $b$ are in \emph{direct-communication} provided either $a\rightarrow b$ or $b\rightarrow a$.  If there is a finite chain of points $c_1,c_2$, etc, leading from $a$ to $b$,  such that $c_1$ is in direct-communication with $c_2$,  $c_2$ with $c_3$, etc,  then $a$ and $b$ are said to be in \emph{indirect communication}.  Clearly,  (indirect) communication induces an equivalent relation on $\bR$: each point communicates with itself; if $a$ and $b$ communicate,  then so do $b$ and $a$; if $a$ and $b$ communicate and also $b$ and $c$,  then so does $a$ and $c$.  Therefore $\bR$ splits into \emph{non-communicating classes}:
\begin{equation}\label{eq:R}
	\bR=\bigcup_{k\in K} J_k,
\end{equation}
where $\{J_k:k\in K\}$ are disjoint,  the class $J_k$ is either a singleton or an interval, and there are at most countably many intervals in this split.  When $J_k$ is a singleton, clear it is $\{x\}$ for some $x\in \Lambda_t$.  Note that $J_k$ is an invariant set of $X$  in the sense that when starts from $J_k$, $X$ will stay in $J_k$ until the lifetime $\zeta$.  Hence the restriction of $X$ to $J_k$ is still a nice diffusion process. 
Let $\mathring{J}_k$ be the interior of $J_k$ and particularly $\mathring{J}_k=\emptyset$ when it is a singleton.  Set 
\begin{equation}\label{eq:RINGR}
	\mathring{\bR}:=\bigcup_{k\in K} \mathring{J}_k,
\end{equation}
which is a union of at most countably many disjoint open intervals.


\begin{theorem}\label{THM0}
Hunt's hypothesis (H) holds for $X$, if and only if  $\mathring{\bR}$ contains only reflecting shunt points in the sense that
\begin{equation}\label{eq:3}
	\mathring{\bR}\cap \Lambda_{pr}=\{x\in \mathring{\bR}\cap \Lambda_{pr}: x=a_n \text{ for some } (a_n,b_n)\text{ in \eqref{eq:lambda2} s.t. }(a_n,b_n)\ni y\rightarrow a_n\}
\end{equation}
and 
\begin{equation}\label{eq:4}
\mathring{\bR}\cap \Lambda_{pl}=\{x\in \mathring{\bR}\cap \Lambda_{pl}: x=b_n \text{ for some } (a_n,b_n)\text{ in \eqref{eq:lambda2} s.t. }(a_n,b_n)\ni y\rightarrow b_n\}.
\end{equation}
\end{theorem}
\begin{proof}
We first show that if \eqref{eq:3} does not hold,  then (H) fails for $X$.  In fact, take $x\in \Lambda_{pr}\cap \mathring{J}_k$ for a certain $k\in K$.  There appear two possible cases:
\begin{description}
\item[\rm (r1)] $\exists \mathring{J}_k\cap \Lambda_{pr}\ni x_n \downarrow x$:  In this case,  it follows from Lemma~\ref{LM1}~(2) that $\mathbf{P}_x(\sigma_x=0)\leq \mathbf{P}_x(\sigma_x<\infty)=0$.  Thus $\{x\}$ is a thin set but $y\rightarrow x$ for some $\mathring{J}_k\ni y<x$ due to Lemma~\ref{LM25}~(2) and the definition of $\mathring{J}_k$.   This means  that $\{x\}$ is not polar, and (H) fails for $X$. 
\item[\rm (r2)] $x=a_n$ for some $n$ but $(a_n,b_n)\ni y\nrightarrow a_n$:  In this case Lemma~\ref{LM1}~(3) implies that $\{x\}$ is thin and similar to (r1), we have that $\{x\}$ is not polar. Hence (H) fails for $X$.  
\end{description}
 Analogically we can show that if \eqref{eq:4} does not hold, then (H) fails for $X$.  

To the contrary, it suffices to show that (H) holds for $X$ under the conditions \eqref{eq:3} and \eqref{eq:4}.  Take a thin set $A$.  Note that every subset of $A$ is also thin, and particularly $\{x\}$ is thin for any $x\in A$.   However, every singleton in $\Lambda_2$ or $\Lambda_t$ is not thin,  due to Lemma~\ref{LM1}~(1) and $\mathbf{P}_x(X_t=x,\forall t)=1$ for $x\in \Lambda_t$.  Hence $A\subset \Lambda_{pr}\cup \Lambda_{pl}$.  We claim that $A\cap \mathring{\bR}=\emptyset$. Argue with contradiction and suppose $x\in A\cap \Lambda_{pr}\cap \mathring{\bR}$. It follows from \eqref{eq:3} that $x=a_n$ for some $n$ and $(a_n,b_n)\ni y\rightarrow x$. Then the subprocess of $X$ killed upon leaving $[a_n,b_n)$ is a regular diffusion in the sense of \cite[(45.1)]{RW87}. Particularly $\{x\}$ is regular for itself.  This contradicts to that $\{x\}$ is thin.  The other case $x\in A\cap \Lambda_{pl}\cap \mathring{\bR}$ can be argued analogically.  Hence $A\cap \mathring{\bR}=\emptyset$ is verified, and it leads to that $A$ is a subset of the family of all closed endpoints of the intervals $J_k$ in \eqref{eq:R}.  Particularly, $A$ is a countable set.  Finally we only need to show that for any $x\in A$,  $\{x\}$ is polar.  To do this,  take $x\in A\cap J_k$ to be the left endpoint of $J_k$.  Then $x\in \Lambda_{pr}$ and $y\nrightarrow x$ for any $y<x$ by the definition of $J_k$.  There appear the following possible cases:
\begin{description}
\item[\rm (r1')] $\exists J_k\cap \Lambda_{pr}\ni x_n \downarrow x$: In this case,  $y\nrightarrow x$ for any $y>x$, because $y>x_n\in \Lambda_{pr}$ for some $n$, and $x\nrightarrow x$ due to Lemma~\ref{LM1}~(2).  In other words,  $\{x\}$ is polar.
\item[\rm (r2')] $x=a_n$ for some $n$ but $(a_n,b_n)\ni y\nrightarrow a_n$: Again $y\nrightarrow x$ for any $y>x$ and $x\nrightarrow x$ due to Lemma~\ref{LM1}~(3).  Hence $\{x\}$ is polar. 
\item[\rm (r3')] $x=a_n$ for some $n$ and $(a_n,b_n)\ni y\rightarrow a_n$: This case is impossible because the subprocess of $X$ killed upon leaving $[a_n,b_n)$ is a regular diffusion, and particularly $\{x\}$ is regular for itself, leading to a contradiction of that $\{x\}$ is thin. 
\end{description}
In a word, $\{x\}$ is polar.  The other case that $x\in A\cap J_k$ is the right endpoint of $J_k$ can be argued analogically.  Eventually we can conclude that $A$ is polar. Therefore (H) holds for $X$.  That completes the proof. 
\end{proof}
\begin{remark}
It is worth noting that the right shunt point $x$ in the case (r1) is the first kind of shunt points breaking (H) as illustrated in Example~\ref{Exa0}~(1), while that in the case (r2) is the second kind of shunt points breaking (H) as illustrated in Example~\ref{Exa0}~(2).  Hence the condition \eqref{eq:3} in fact rules out these two kinds of right shunt points. 
\end{remark}

\section{From Hunt's hypothesis to symmetrization outside $\Lambda_t$}\label{SEC4}

Let us turn to consider the symmetrization for $X$.  The following example, which motivates the current study,  shows that a non-symmetric diffusion on $\bR$ may satisfy (H).  

\begin{example}\label{Exa1}
Let $Y^0$ be a diffusion process on $\bR$ consisting of two distinct components: The restriction of $Y^0$ to $(-\infty, 0)$ is an absorbing Brownian motion ($0$ is the absorbing boundary) and the restriction to $[0,\infty)$ is a reflecting Brownian motion.  

By a Ikeda-Nagasawa-Watanabe piecing-out construction as in Example~\ref{Exa0}~(2), we can obtain a conservative diffusion process $Y$ on $\bR$, and particularly, for all $x<0$ and $y\geq 0$, $x\rightarrow y$ but $y\nrightarrow x$.  Clearly,  $0$ is a right shunt point and $Y$ is not symmetrizable due to Lemma~\ref{LM1}~(4).  However, (H) holds for $Y$ since a thin set must be empty. 
\end{example}

Generally, we take the following family of shunt points playing similar roles to $0$ in this example: 
\begin{equation}\label{definitionofLambdaa}
\Lambda_{ap}:=\{x\in \Lambda_{pl}\cup \Lambda_{pr}: \exists (a,b)\ni x\text{ s.t. }(a,x)\cup (x, b)\subset \Lambda_2 \text{ and } a\rightarrow x, b\rightarrow x\}. 
\end{equation}
In other words, every $x\in \Lambda_{ap}$ is not only a shunt point but also the common reachable endpoint of two regular intervals.  Note that given an interval $(a,b)$,  the part process $\hat{X}$ of $X$ on $(a,b)$ is defined as
\begin{equation}\label{eq:hatXt}
	\hat{X}_t:=\left\lbrace\begin{aligned}
	&X_t,\quad t<\tau:=\{t>0:X_t\notin (a,b)\}, \\
	&\Delta,\quad t\geq \tau. 
	\end{aligned} \right.
\end{equation}
Analogical to Example~\ref{Exa1},  the lemma below shows that (H) holds but the symmetrization fails near a point in $\Lambda_{ap}$. 

\begin{lemma}\label{LM2}
Let $x\in \Lambda_{ap}$ and $(a,b)$ be such an interval in \eqref{definitionofLambdaa}.  Then (H) holds but the symmetrization fails locally at $x$ in the sense that 
\begin{itemize}
\item[(1)] The part process $\hat{X}$ of $X$ on $(a,b)$ satisfies (H);
\item[(2)] $\hat{X}$ is not symmetrizable.  
\end{itemize}
\end{lemma}
\begin{proof}
\begin{itemize}
\item[(1)] This is clear by Theorem~\ref{THM0}.
\item[(2)] Mimicking the proof of \cite[Lemma~3.2]{L19}, we can conclude that if $\hat{X}$ is symmetric with respect to a certain fully supported $\sigma$-finite measure, then $y\nrightarrow x$ for any $y\in (a,x)$. This contradicts to the definition of $\Lambda_{ap}$. That completes the proof. 
\end{itemize}
\end{proof}

Denote $E:=\Lambda_t^c$, and define a subprocess $X^0$ of $X$ killed  upon hitting $\Lambda_t$:
\begin{equation}\label{eq:X0t}
X^0_t:=\left\lbrace 
\begin{aligned}
& X_t,\quad t<\zeta^0:=\zeta\wedge \sigma_{\Lambda_t}, \\
&\Delta, \quad t\geq \zeta^0.  \end{aligned}
\right.
\end{equation}
Denote the family of all fully supported $\sigma$-finite positive measures on $E$ (resp. $\bR$) by $\mathscr{M}_E$ (resp. $\mathscr{M}_\bR$). 
Now we have a position to state our main result. 

\begin{theorem}\label{THM1}
Let $X^0$ be the subprocess of $X$ on $E$ as above. Then the following are equivalent: 
\begin{itemize}
\item[(1)] $\Lambda_{ap}=\emptyset$ and (H) holds for $X$;
\item[(2)] $\Lambda_{ap}=\emptyset$ and ($\text{H}_\xi$) holds for $X$ with one (or equivalently all) $\xi\in \mathscr{M}_\bR$; 
\item[(3)] $X^0$ is symmetrizable on $E$, i.e. there exists $m^0\in \mathscr{M}_E$ such that $X^0$ is an $m^0$-symmetric diffusion process on $E$. 
\end{itemize}
\end{theorem}
\begin{proof}
\emph{(3)$\Rightarrow$(2).}  Suppose (3) holds.  If $x\in \Lambda_{ap}$, then $\hat{X}$ in \eqref{eq:hatXt} is still $m^0|_{(a,b)}$-symmetric due to the symmetry of $X^0$. This leads to a contradiction of Lemma~\ref{LM2}~(2).  Thus $\Lambda_{ap}=\emptyset$. On the other hand, the associated Dirichlet form of $X^0$ on $L^2(E,m^0)$ is necessarily quasi-regular by means of, e.g., \cite[Theorem~1.5.3]{CF12},  and particularly, ($\text{H}_{m^0}$) holds for $X^0$ due to \cite[Theorem~3.1.10]{CF12}.  Set $\xi|_{E}:=m^0$ and $\xi|_{E^c}:=dx|_{E^c}$.  We claim that ($\text{H}_\xi$) holds for $X$. Take a thin set $A$ with respect to $X$.  Since $\{x\}$ is not thin with respect to $X$ for any $x\in \Lambda_t$, it follows that $A\subset E$ and $A$ is also thin with respect to $X^0$.  This yields by ($\text{H}_{m^0}$) that $A$ is $m^0$-polar with respect to $X^0$. Therefore $A$ is $\xi$-polar with respect to $X$, because all points in $E^c$ are traps.  In other words, (2) is satisfied. 

\emph{(2)$\Rightarrow$(3).} Now suppose (2) holds.  
Set $K:=\Lambda^c_2$, which is closed.  The derivation of (3) will be completed in several steps.  


Firstly we claim that if an open interval $(a,b)\subset K$, then $(a,b)\subset \Lambda_t$.  Argue by contradiction and suppose $x\in (a,b)\cap \Lambda_{pr}$.  The other case $x\in (a,b)\cap \Lambda_{pl}$ can be treated analogically.  By Lemma~\ref{LM25}~(4, 5),  there exists a constant $\varepsilon>0$ such that 
\[
	(x,x+\varepsilon)\subset \left(\Lambda_{pr}\cup \Lambda_2\right)\cap K=\Lambda_{pr}
\]
and $x\rightarrow y$ for any $y\in (x,x+\varepsilon)$.  Then $\mathbf{P}_y(\sigma_y=0)\leq \bfP_y(\sigma_y<\infty)=0$ by Lemma~\ref{LM1}~(2) and thus $\{y\}$ is thin.  By (H${}_\xi$),  $\{x\}$ is also $\xi$-polar. However $z\rightarrow y$ for any $z\in (x,y)$ due to $x\rightarrow y$.  This leads to a contradiction. 

Secondly, we show that for any $x\in \Lambda_{pr}$ (resp. $x\in \Lambda_{pl}$),  $y\nrightarrow x$ for any $y<x$ (resp. $y>x$).  Only the case $x\in \Lambda_{pr}$ will be treated as below.  Note that $(x,x+\varepsilon)\subset \Lambda_{pr}\cup \Lambda_2$ for a certain $\varepsilon>0$.  When a sequence $x_n$ in $\Lambda_{pr}$ decreases to $x$,  it follows from Lemma~\ref{LM1}~(2) that $\{x\}$ is thin and thus $\xi$-polar by (H${}_\xi$).  Since $\xi$ is fully supported, it is easy to conclude that $y\nrightarrow x$ for any $y<x$.  Now we only need to consider the case that $x$ is the left endpoint of a regular interval in \eqref{eq:lambda2}, i.e. $x=a_n$ for a certain $1\leq n\leq N$.  The argument will be completed in various cases:
\begin{itemize}
\item[(a)] \emph{There exists a sequence $x_n\in \Lambda_l$ increasing to $x$. }
 In this case, for any $y<x$, take $y<x_n<x$. It follows from Lemma~\ref{LM25}~(3) that $\mathbf{P}_y(X_t\leq x_n,\forall t)=1$ and thus $y\nrightarrow x$.
\item[(b)] \emph{$(x-\varepsilon,x)\subset \Lambda_2$ for a certain $\varepsilon>0$. }
In this case, when $(x-\varepsilon,x)\ni y\rightarrow x$ and $(a_n,b_n)\ni z\rightarrow x$, we have $x\in \Lambda_{ap}$ contradicting to $\Lambda_{ap}=\emptyset$.  It suffices to show $y\nrightarrow x$ if $z\nrightarrow x$.  Now suppose $z\nrightarrow x$.  Then from Lemma~\ref{LM1}~(3), we know that $\mathbf{P}_a(\sigma_a=0)\leq \mathbf{P}_a(\sigma_a<\infty)=0$.  In other words, $\{a\}$ is thin and hence $\xi$-polar by (H${}_\xi$). Particularly, $y'\nrightarrow x$ for any $y'<x$. 
\item[(c)] \emph{$(x-\varepsilon,x)\subset \Lambda_{pr}\cup \Lambda_2$ for a certain $\varepsilon>0$ and there exists a sequence $x_n\in \Lambda_{pr}$ increasing to $x$. } 
For any $y<x$, take $y<x_n<x$.  Clearly a certain regular interval is contained in $(x_n,x)$, i.e.  $(a_m,b_m)\subset (x_n,x)$ for a certain $1\leq m\leq N$. Then $b_m\in \Lambda_{pr}$, and  $b_m$ is either the left endpoint of another regular interval or the decreasing limit of a sequence in $\Lambda_{pr}$.  For either case, we can obtain that $y\nrightarrow b_m$ analogically to (b) or the case $\Lambda_{pr}\ni x_n\downarrow x$.  As a consequence, $y\nrightarrow x$. 
\end{itemize}

Thirdly,  $\Lambda_{pr}\subset\{a_n: 1\leq n\leq N\}$ and $\Lambda_{pl}\subset \{b_n:1\leq n\leq N\}$.  In fact, take $x\in \Lambda_{pr}$.  If $x$ is not a left endpoint of a regular interval,  there must exist a sequence $\Lambda_{pr}\ni x_n\downarrow x$.  Then  $x\nrightarrow x_n$ for all $n$ due to the conclusion in the second step.  This contradicts to Lemma~\ref{LM25}~(4).  The case $x\in \Lambda_{pl}$ can be treated analogically. 

Fourthly set an interval $I_n:=\langle a_n,b_n\rangle$ for every $1\leq n\leq N$, where $a_n\in I_n$ (resp. $b_n\in I_n$) if and only if $a_n\in \Lambda_{pr}$ (resp. $b_n\in \Lambda_{pl}$) and $x\rightarrow a_n$ (resp. $x\rightarrow b_n$) for $x\in (a_n,b_n)$.  It is easy to figure out that $\{I_n\}$ are disjoint,  $I_n$ is an invariant set of $X^0$ in the sense that $\mathbf{P}_x(X^0_t\in I_n,\forall t<\zeta^0)=1$ for all $x\in I_n$, and the restriction $X^n$ of $X^0$ to $I_n$ is a regular diffusion in the sense of \cite[(45.1)]{RW87}.  Hence there exist a so-called scale function $s_n$ and a so-called speed measure $m_n$ characterizing $X^n$.  Particularly, $m_n$ is a fully supported Radon measure on $I_n$ and $X^n$ is symmetric with respect to $m_n$. 

Finally, set another interval $\tilde{I}_n:=\langle a_n,b_n\rangle$ for every $1\leq n\leq N$, where $a_n\in \tilde I_n$ (resp. $b_n\in \tilde I_n$) if and only if $a_n\in \Lambda_{pr}$.  Clearly $\{\tilde{I}_n\}$ are disjoint, and $\tilde{I}_n$ is still an invariant set of $X^0$. In addition,  $E=\cup_{1\leq n\leq N} \tilde{I}_n$ and $X^0$ consists of a sequence of distinct components, the restrictions $\tilde{X}^n$ to all $\tilde{I}_n$.  Extending $m_n$ to $\tilde I_n$ by imposing $m_n(\tilde{I}_n\setminus I_n)=0$, we can easily obtain that $\tilde{X}^n$ is still symmetric with respect to $m_n$. Set $m^0:=m_n$ on $\tilde{I}_n$ for all $1\leq n\leq N$.  Clearly $m^0\in \mathscr{M}_E$.  Eventually we can conclude that $X^0$ is an $m^0$-symmetric diffusion process on $E$. 

\emph{(3)$\Rightarrow$(1).}  $\Lambda_{ap}=\emptyset$ can be obtained analogically to the proof of \emph{(3)$\Rightarrow$(2)}.  On the other hand, from the derivation of \emph{(2)$\Rightarrow$(3)}, we know that $$\mathring{\bR}=\cup_{n=1}^N (a_n,b_n)$$ contains no shunt points.  Hence (H) holds for $X$ by virtue of Theorem~\ref{THM0}. 


\emph{(1)$\Rightarrow$(2)} is obvious.  The proof is eventually completed. 
\end{proof}
\begin{remark}
We need to point out that $\Lambda_t$ is (nearly) Borel measurable so that $\sigma_{\Lambda_t}$ is well defined in \eqref{eq:X0t}.  In fact,  by the argument of (2)$\Rightarrow$(3),  one can see that in our consideration,  $E=\cup_{1\leq n\leq N}\tilde{I}_n$, where $\tilde{I}_n$ are intervals.  Hence $\Lambda_t=E^c$ is obviously Borel measurable.  
\end{remark}

From this proof, one can find that under each condition of Theorem~\ref{THM1},  $X^0$ is, in fact, a distinct union of regular diffusions $X^n$ on $I_n=\langle a_n, b_n\rangle$, obtained by adding reachable shunt endpoints to the regular interval $(a_n,b_n)$. There may appear entrance boundary points in $\tilde{I}_n\setminus I_n$, where $\tilde{I}_n$ contains all shunt endpoints of $(a_n,b_n)$, for $X^n$, while $\tilde{I}_n\setminus I_n$ is polar and thus can be essentially ignored.  The endpoint $a_n$ (resp. $b_n$) is called \emph{exit} for $X^n$ provided
\[
	a_n\notin I_n,  I_n \ni x\rightarrow a_n,\quad \text{ (resp.  }b_n\notin I_n,  I_n \ni x\rightarrow b_n). 
\]
We emphasize that a finite exit boundary point $a_n$ or $b_n$ of $X^n$ must be a trap of $X$, i.e. $a_n\in \Lambda_t$ or $b_n\in \Lambda_t$. 

\begin{corollary}
If $a_n$ (resp. $b_n$) is finite exit for $X^n$, then $a_n\in \Lambda_t$ (resp. $b_n\in \Lambda_t$). 
\end{corollary}
\begin{proof}
We only treat the case of $a_n$.  Argue by contradiction. If $a_n\in \Lambda_{pr}$, then $x\rightarrow a_n$ leads to $a_n\in I_n$ by the definition of $I_n$, which contradicts to $a_n\notin I_n$.  If $a_n\in \Lambda_{pl}$, then the second step in the proof of (2)$\Rightarrow$(3) of Theorem~\ref{THM1} tells us that $x\nrightarrow a_n$, also leading to a contradiction.  Therefore we must have $a_n\in \Lambda_t$. 
\end{proof}

When $a_n$ or $b_n$ is finite exit for $X^n$,  
\[
	\hat I_n:= I_n\cup \{a_n\} \text{ or }I_n\cup \{b_n\}
\] is also an invariant set of $X$ and the restriction $\hat{X}^n$ of $X$ to $\hat{I}_n$ is a nice diffusion process. But $\hat{X}^n$ is not symmetrizable due to $a_n\nrightarrow x\in I_n$ or $b_n\nrightarrow x\in I_n$.  This is the reason why we only consider the subprocess $X^0$ obtained by killing upon hitting $\Lambda_t$ in Theorem~\ref{THM1}.  A simple example blew illustrates this fact. 

\begin{example}\label{Exa2}
Let $Y^0$ be a diffusion process on $[0, \infty)$, consisting of two distinct components: The restriction of $Y^0$ to $(0,\infty)$ is an absorbing Brownian motion, killed upon hitting $0$, and the restriction to $\{0\}$ is the trivial process $Y^0_t\equiv 0$ for all $t\geq 0$.  
By a Ikeda-Nagasawa-Watanabe piecing-out construction with the instantaneous distribution $\delta_0$ like that in Example~\ref{Exa1},  we can obtain a conservative diffusion process $Y$ on $[0,\infty)$, and particularly,  $x\rightarrow 0$ for all $x>0$ but $0$ is a trap. 

Let $\zeta^{Y^0}$ and $\zeta^Y$ be the lifetimes of $Y^0$ and $Y$ respectively. 
We should emphasize that when $\zeta^{Y^0}<\infty$,  $Y^0_{\zeta^{Y^0}-}=0$ while $Y^0_t=\Delta$, the ceremony,  for all $t\geq \zeta^{Y^0}$.  As a comparison $\zeta^Y$ is always infinite,  i.e. $\zeta^Y=\infty$, and when $t\geq \sigma^Y_0:=\inf\{t>0:Y_t=0\}$,  it holds $Y_t=0$. 

The subprocess of $Y$ killed upon hitting $0$, i.e. the absorbing Brownian motion on $(0,\infty)$,  is obviously symmetric with respect to the Lebesgue measure.  Nevertheless, $Y$ is not symmetrizable, since the symmetry of $Y$ must imply that all traps are unreachable by other points, which is a contradiction of $x\rightarrow 0$ for $x>0$; see Lemma~\ref{LM1}~(4). 
\end{example}

\section{Symmetrization on $\bR$}\label{SEC5}

This section is devoted to exploring when $X$ is symmetrizable on $\bR$.  In Example~\ref{Exa2}, we show an example that when a finite exit boundary point appears for certain $X^n$,  $X$ is symmetrizable only outside $\Lambda_t$ but not on $\bR$.  To state a general result, set 
\[
	\Lambda_{at}:=\{x\in \Lambda_t: \exists y\neq x \text{ s.t. }y\rightarrow x\}
\]
to be the family of all traps that can be reached by other points.  The following theorem claims that on basis of the symmetrization outside $\Lambda_t$, $X$ is symmetrizable on $\bR$, if and only if there appear no finite exit boundary points. 

\begin{theorem}\label{THM2}
The following are equivalent:
\begin{itemize}
\item[(1)] $\Lambda_{ap}=\Lambda_{at}=\emptyset$ and (H) holds for $X$;
\item[(2)] $\Lambda_{ap}=\Lambda_{at}=\emptyset$ and ($H_\xi$) holds for $X$ with one (or equivalently all) $\xi\in \mathscr{M}_\bR$; 
\item[(3)] $X$ is symmetrizable on $\bR$, i.e. there exists $m\in \mathscr{M}_\bR$ such that $X$ is an $m$-symmetric diffusion process on $\bR$. 
\end{itemize}
\end{theorem}
\begin{proof}
\emph{(2)$\Rightarrow$(3)}.  From the argument of (2)$\Rightarrow$(3) in the proof of Theorem~\ref{THM1},  we can figure out that $X^n$, the restriction of $X$ to $I_n$,  is symmetric with respect to $m_n$.  Set a $\sigma$-finite measure $m$ on $\bR$ as follows: $m|_{I_n}:=m_n$ for all $1\leq n\leq N$ and $$m|_{\left(\cup_{1\leq n\leq N}I_n\right)^c}:=dx|_{\left(\cup_{1\leq n\leq N}I_n\right)^c}.$$
Clearly $m$ is fully supported. Let $(P_t)$ and $(P^n_t)$ be the semigroups of $X$ and $X^n$ respectively.  It suffices to show $(P_t)$ is $m$-symmetric. Note that for any positive Borel measurable function $f$ on $\bR$, 
\[
	P_tf(x)=\left\lbrace \begin{aligned}
	 & P^n_t(f|_{I_n})(x),\quad x\in I_n,  1\leq n\leq N,\\
	 &f(x),\quad x\in \Lambda_t,
	\end{aligned} \right.
\]
and $N:= \left(\left(\cup_{1\leq n\leq N}I_n\right) \cup \Lambda_t \right)^c=\cup_{1\leq n\leq N} (\tilde{I}_n\setminus I_n)$ is of zero $m$-measure.  By virtue of the symmetry of $P^n_t$ with respect to $m_n$, it is straightforward to verify that $(P_t)$ is symmetric with respect to $m$. 

\emph{(3)$\Rightarrow$(1)}.  Mimicking the proof of Theorem~\ref{THM1}, we only need to show the additional condition $\Lambda_{at}=\emptyset$.  In fact, this is a consequence of Lemma~\ref{LM1}~(4). 

\emph{(1)$\Rightarrow$(2)}. This is obvious. The proof is eventually completed. 
\end{proof}
\begin{remark}
When the equivalent conditions in Theorem~\ref{THM2} hold,  the possible exit boundary point for $X^n$ has to be $\infty$ or $-\infty$. 
\end{remark}

It is worth noting that a regular diffusion on an interval has unique symmetric measures up to a constant, see, e.g., \cite{YZ10}.  In other words,  the family of all $\sigma$-finite symmetric measures of $X^n$ is $\{c\cdot m_n: c>0\}$.  As a consequence, one can figure out the following expression of all symmetric measures of $X$.  The proof is trivial and we omit it. 

\begin{corollary}\label{COR1}
Under each of the equivalent conditions in Theorem~\ref{THM2},  the family of all fully supported $\sigma$-finite symmetric measures of $X$ is
\[
\mathcal{M}=\left\{m\in \mathscr{M}_\bR: \exists c_1,\cdots, c_N>0\text{ s.t. }m|_{I_n}=c_n\cdot m_n\text{ for }1\leq n\leq N\right\}.
\]
\end{corollary}

\section{Regularity of associated Dirichlet form}\label{SEC6}

We have noted in the proof of Theorem~\ref{THM1} that $X^n$ is determined by a so-called scale function $s_n$, i.e. a strictly increasing and continuous function on $I_n$, and a speed measure $m_n$.  To be precise,  it is associated with a regular Dirichlet form on $L^2(I_n, m_n)$:
\begin{equation}\label{eq:Dirichletforms}
\begin{aligned}
	\sF^n&:=\bigg\{f\in L^2(I_n,m_n): f\ll s_n,  \frac{df}{ds_n}\in L^2(I_n,ds_n), \\ 
	 &\qquad\qquad  f(a_n)=0\text{ (resp. }f(b_n)=0\text{) whenever }a_n \text{ (resp.  }b_n)\text{ is exit}\bigg\},\\
	 \sE^n&(f,g):=\frac{1}{2}\int_{I_n} \frac{df}{ds_n}\frac{dg}{ds_n}ds_n,\quad f,g\in \sF^n, 
\end{aligned}
\end{equation}
where $f\ll s_n$ means that $f$ is absolutely continuous with respect to $s_n$,  $ds_n$ is the Lebesgue-Stieltjes measure induced by $s_n$, and $f(a_n):=\lim_{x\downarrow a_n}f(x)$ (resp.  $f(b_n):=\lim_{x\uparrow b_n}f(x)$) is well defined for the functions under consideration; see, e.g., \cite{LY19}.  Since $a_n$ is strictly increasing and continuous, 
\[
 s_n(a_n):=\lim_{x\downarrow a}s(x),\quad s_n(b_n):=\lim_{x\uparrow b_n}s_n(x)
\]
is well defined in $\bar{\bR}=[-\infty, \infty]$ no matter $a_n, b_n$ belong to $I_n$ or not. 
In addition,  $a_n$ (resp. $b_n$) is exit  for $X^n$ (or called \emph{approachable in finite time} for $X^n$ as in \cite[(3.5.13)]{CF12}), if and only if 
\[
	\int_{a_n}^c m_n\left((x,c) \right)ds_n(x)<\infty\quad \left(\text{resp. }\int^{b_n}_c m_n\left((c,x) \right)ds_n(x)<\infty\right),
\]
where $c\in (a_n,b_n)$ is arbitrarily chosen.

Mimicking \cite[Theorem~3.2]{L19},  we can obtain the associated Dirichet form of $X^0$, a distinct union of those associated with $X^n$. The proof is trivial and we omit it.

\begin{lemma}
Assume that the equivalent conditions in Theorem~\ref{THM1} hold.  Let $m^0|_{I_n}:=m_n$ for $1\leq n\leq N$ and $m^0|_{E\setminus \cup_{1\leq n\leq N} I_n}:=0$.  Then $X^0$ is associated with the quasi-regular Dirichlet form on $L^2(E,m^0)$:
\begin{equation}\label{eq:Dirichletfom}
	\begin{aligned}
		&\sF^0:=\{f\in L^2(E,m^0): f|_{I_n}\in \sF^n, 1\leq n\leq N\},  \\
		&\sE^0(f,g):=\sum_{n=1}^N \sE^n(f|_{I_n},g|_{I_n}),\quad f,g\in \sF^0,
	\end{aligned}
\end{equation}
where $(\sE^n,\sF^n)$ is given by \eqref{eq:Dirichletforms}. 
\end{lemma}

Finally we turn to the case $X$ is symmetrizable on $\bR$. Take $m\in \mathcal{M}$ in Corollary~\ref{COR1} and assume without loss of generality that $c_1=\cdots =c_N=1$.  The associated Dirichlet form of $X$ on $L^2(\bR,m)$ can be expressed analogically as
\[
	\begin{aligned}
		&\sF:=\{f\in L^2(\bR,m): f|_{I_n}\in \sF^n, 1\leq n\leq N\},  \\
		&\sE(f,g):=\sum_{n=1}^N \sE^n(f|_{I_n},g|_{I_n}),\quad f,g\in \sF. 
	\end{aligned}
\]
Clearly, $(\sE,\sF)$ is quasi-regular but not necessarily regular on $L^2(\bR,m)$, because $m$ is probably not Radon.  We end this section with a result to show when $(\sE,\sF)$ is regular. 

\begin{theorem}
Let $m\in \mathcal{M}$ and assume that the equivalent conditions in Theorem~\ref{THM2} hold.  Then the associated Dirichlet form $(\sE,\sF)$ of $X$ is regular on $L^2(\bR,m)$, if and only if $m$ is Radon on $\bR$. 
\end{theorem}
\begin{proof}
The necessity is clear.  Now suppose that $m$ is Radon.  We will employ \cite[Corollary~2.13]{LY19} to prove the regularity of $(\sE,\sF)$.  Indeed,  it suffices to show that $s_n$ is adapted to $I_n$ for all $1\leq n\leq N$ in the sense of \cite[(2.21)]{LY19}. In other words,  we need to verify that when $a_n>-\infty$ (resp. $b_n<\infty$),
\begin{equation}\label{eq:33}
	a_n\in I_n \text{ (resp. } b_n\in I_n), \text{ if and only if }s_n(a_n)>-\infty \text{ (resp. }s_n(b_n)<\infty). 
\end{equation}
Only the case $a_n$ will be treated and the other case is analogical.  In practise, $a_n\in I_n$ amounts to $x\rightarrow a_n\in \Lambda_{pr}$ for $x\in (a_n,b_n)$ by the definition of $I_n$.  In other words,  $a_n$ is approachable in finite time for the part process $X^{n,0}$ of $X^n$ on $(a_n,b_n)$.  It follows from \cite[(3.5.13)]{CF12} that by fixing $c\in (a_n,b_n)$, 
\begin{equation}\label{eq:32}
	\int_{a_n}^c m_n\left((x,c) \right)ds_n(x)<\infty. 
\end{equation}
Take $c'\in (a_n,c)$ and the left hand side of  \eqref{eq:32} is greater than
\[
	\int_{a_n}^{c'} m_n\left((x,c) \right)ds_n(x)\geq m_n\left((c',c) \right)\cdot \left( s_n(c')-s_n(a_n)\right). 
\]
Note that $0<m_n\left((c',c) \right)<\infty$ and $s_n(c')$ is finite.  Hence we can conclude that $s_n(a_n)>-\infty$.  To the contrary,  suppose $a_n\notin I_n$.  This amounts to $(a_n,b_n)\ni x\nrightarrow a_n$ or $a_n\notin \Lambda_{pr}$.  The latter case implies $a_n\in \Lambda_{pl}\cup \Lambda_t$, also leading to $x\nrightarrow a_n$ because of the second step to prove (2)$\Rightarrow$(3) in Theorem~\ref{THM1} and $\Lambda_{at}=\emptyset$.  This means that $a_n$ is not approachable in finite time for $X^{n,0}$.  By using  \cite[(3.5.13)]{CF12} again, we have
\[
	\int_{a_n}^c m_n\left((x,c) \right)ds_n(x)=\infty. 
\]
Since $m_n\left((x,c) \right)\leq m\left([a_n,c] \right)<\infty$ and $s_n(c)$ is finite,  it follows that $s_n(a_n)=-\infty$.  Therefore \eqref{eq:33} is concluded and the proof is eventually completed. 
\end{proof}

\bibliographystyle{abbrv}
\bibliography{H-hypothesis}


\end{document}